\documentclass[12pt,twoside]{amsart}
\usepackage{amsfonts,amsmath}
\usepackage{pinlabel}
\usepackage{enumerate}
\usepackage{mathtools}
\usepackage{xcolor,colortbl}
\usepackage{url}
\usepackage{cleveref}

\def\endpf{\relax\ifmmode\expandafter\endproofmath\else
  \unskip\nobreak\hfil\penalty50\hskip.75em\hbox{}\nobreak\hfil\bull
  {\parfillskip=0pt \finalhyphendemerits=0 \bigbreak}\fi}
\def\bull{\vbox{\hrule\hbox{\vrule\kern3pt\vbox{\kern6pt}\kern3pt\vrule}\hrule}}

\newtheorem{defn}{Definition}[section]
\newtheorem{lemma}[defn]{Lemma}

\newtheorem{remark}[defn]{Remark}
\newtheorem{proposition}[defn]{Proposition}

\newtheorem{maintheorem}{Theorem}

\newtheorem{example}[defn]{Example}

\newcommand{\zz}{{\mathbb Z}}

\newcommand{\nn}{{\mathbb N}}
\newcommand{\qq}{{\mathbb Q}}

\newcommand{\spin}{\ifmmode{\rm Spin}\else{${\rm spin}$\ }\fi}
\newcommand{\spinc}{\ifmmode{{\rm Spin}^c}\else{${\rm spin}^c$\ }\fi}

\newcommand{\tors}{{\it Tors}}

\newcommand{\cali}{\mathcal{I}}

\newcommand{\CP}{\mathbb{CP}}
\newcommand{\PP}{\mathbb{P}}

\definecolor{Gray}{gray}{0.8}

\newenvironment{narrow}[2]{%
 \begin{list}{}{%
  \setlength{\topsep}{0pt}%
  \setlength{\leftmargin}{#1}%
  \setlength{\rightmargin}{#2}%
  \setlength{\listparindent}{\parindent}%
  \setlength{\itemindent}{\parindent}%
  \setlength{\parsep}{\parskip}%
 }%
\item[]}{\end{list}}

\newif\ifpic
\picfalse    

\DeclareMathOperator{\Aut}{Aut}

\begin{document}

\title{Smooth, nonsymplectic embeddings of rational balls in the complex projective plane}
\author[Brendan Owens]{Brendan Owens}
\address{School of Mathematics and Statistics \newline\indent 
University of Glasgow \newline\indent 
Glasgow, G12 8QQ, United Kingdom}
\email{brendan.owens@glasgow.ac.uk}
\date{\today}
\thanks{}

\begin{abstract}  We exhibit an infinite family of rational homology balls which embed smoothly but not symplectically in the complex projective plane.  We also obtain a new lattice embedding obstruction from Donaldson's diagonalisation theorem, and use this to show that no two of our examples may be embedded disjointly.
\end{abstract}

\maketitle

\pagestyle{myheadings}
\markboth{BRENDAN OWENS}{NONSYMPLECTIC EMBEDDINGS OF RATIONAL HOMOLOGY BALLS}


\section{Introduction}
\label{sec:intro}

A Markov triple is a positive integer solution $(p_1,p_2,p_3)$ to the Markov  equation
\begin{equation}\label{eq:Mkv}p_1^2+p_2^2+p_3^2=3p_1p_2p_3.
\end{equation}
Each Markov triple gives rise to an embedding
\begin{equation}\label{eq:emb}
\bigsqcup_{i=1}^3 B_{p_i,\,q_i}\hookrightarrow\CP^2
\end{equation}
of a disjoint union of three rational homology balls in the complex projective plane.  Here $B_{p,q}$ is the rational homology ball smoothing of the quotient singularity $\frac1{p^2}(1,pq-1)$.  The embedding in \eqref{eq:emb} arises by smoothing the three singular points in the weighted projective space $\PP(p_1^2,p_2^2,p_3^2)$,
and the numbers $q_i$ are given by
$$q_i=\pm3p_j/p_k\pmod{p_i},$$ where $i,j,k$ is a permutation of $1,2,3$.  The apparent sign ambiguity here is due to the fact that $B_{p,q}\cong B_{p,p-q}$.

Hacking and Prokhorov proved in \cite{hp} that any projective surface with quotient singularities which admits a smoothing to $\CP^2$ is $\qq$-Gorenstein deformation equivalent to some $\PP(p_1^2,p_2^2,p_3^2)$ as above. Evans and Smith proved in \cite{es} that any disjoint union $\bigsqcup_{i\in\cali}B_{p_i,\,q_i}$ which admits a symplectic embedding in $\CP^2$  arises in this way, with $|\cali|\le3$.

Let $F(2n-1)$ denote the $n$th odd Fibonacci number, defined by the recursion
\begin{equation}\label{eq:Fib}
F(2n+3)=3F(2n+1)-F(2n-1),\quad F(1)=1,\quad F(3)=2.
\end{equation}
Then $(1,F(2n-1),F(2n+1))$ is a Markov triple for each $n\in\nn$, showing in particular that
$B_{F(2n+1),F(2n-3)}$ admits a symplectic embedding in $\CP^2$ for each $n>1$.

In \cite{balls} we mentioned but overlooked the significance of the following result.  Here $\Delta_{p,q}$ is a properly embedded 
surface in the 4-ball whose double branched cover is $B_{p,q}$, and $P_+$ is the unknotted M\"{o}bius band in the 4-ball with normal Euler number $2$; see \cite {balls} for further details.

\begin{maintheorem}
\label{thm:CP2}
For each $n\in\nn$, the slice surface $\Delta_{F(2n+1),F(2n-1)}$ admits a simple embedding as a sublevel surface of the unknotted M\"{o}bius band $P_{+}$.  Taking double branched covers yields a simple smooth embedding
$$B_{F(2n+1),F(2n-1)}\hookrightarrow \CP^2.$$
If $n>1$, then $B_{F(2n+1),F(2n-1)}$ does not embed symplectically in $\CP^2$.
\end{maintheorem}

\Cref{thm:CP2} gives the first-known smooth embeddings of rational balls $B_{p,q}$ in the complex projective plane that do not arise from symplectic embeddings.  This shows that the smooth embedding problem has an as-yet-unknown solution which differs from that to the symplectic problem solved by Evans-Smith.  Bulent Tosun has informed the author that work of Nemirovski-Segal \cite{NS} implies the existence of a rational ball, bounded by a Seifert fibred space with 3 exceptional fibres, which embeds smoothly but not sympectically in $\CP^2$.  Most of the embeddings obtained in \cite{balls}, but not those given in \Cref{thm:CP2}, have since been reproved and generalised by different methods in \cite{psu}.

A conjecture of Koll\'{a}r \cite{kollar} would imply that at most three rational balls $B_{p_i,\,q_i}$ may embed smoothly and disjointly in $\CP^2$.  The following result gives some mild support to this conjecture.

\begin{maintheorem}
\label{thm:CP2disjoint}
It is not possible to smoothly embed a disjoint union $\bigsqcup_{i\in\cali}B_{p_i,\,q_i}$ of two or more of the balls from \Cref{thm:CP2} in $\CP^2$, where each $(p_i,q_i)$ is a consecutive pair of odd Fibonacci numbers.
\end{maintheorem}

This result uses a new obstruction derived from Donaldson's diagonalisation theorem \cite{don}.  This is stated in \Cref{prop:Don}.

 \
 
\noindent{\bf Corrigendum to \cite{balls}.}   In \cite[sentence after Theorem 5, and Remark 4.1]{balls} we incorrectly stated that $B_{F(2n+1),F(2n-1)}$ embeds symplectically in $\CP^2$.  I am very grateful to Giancarlo Urz\'{u}a who reminded me that the Markov triple $(1,F(2n-1),F(2n+1))$ gives rise to a symplectic embedding in $\CP^2$ of $B_{F(2n+1),F(2n-3)}$, and not of $B_{F(2n+1),F(2n-1)}$.

 \
 
 \noindent{\bf Further acknowledgements.}    I am grateful to Jonny Evans, Marco Golla, Ana Lecuona, Yank\i\ Lekili, Duncan McCoy, Bulent Tosun, and Giancarlo Urz\'{u}a for helpful comments and conversations.  I also thank the anonymous referee for helpful suggestions.


\section{Smooth embeddings}
\label{sec:emb}

In this section we prove \Cref{thm:CP2}, using the method from \cite{balls}.  

We refer the reader to \cite{aigner} for an excellent and readable source on Markov numbers.  Suppose that $(p,a,b)$ is a solution to the Markov equation \eqref{eq:Mkv} with $p>a,b$.  
By \cite[Corollary 3.4]{aigner}, the integers in a Markov triple are pairwise relatively prime, so that there are unique solutions $x=u,u'$ to $$b\equiv \pm xa\pmod{p}.$$
These satisfy $u+u'\equiv 0\pmod p$, so that one of them (say $u$) is between 0 and $p/2$; we call this number $u$ the characteristic number of the Markov triple $(p,a,b)$.
The Markov equation gives $a^2+b^2\equiv 0\pmod p$, from which it follows that $$u^2\equiv -1\pmod p.$$
I am grateful to Jonny Evans for helping me to see the following result.
\begin{lemma}
Let $n\in\nn$.  The rational ball $B_{F(2n+1),F(2n-1)}$ embeds symplectically in $\CP^2$ if and only if $n=1$.
\label{lem:nonsymp}
\end{lemma}
\begin{proof} 
From \cite[Theorem 4.15]{es} we have that $B_{p,q}$ embeds symplectically in $\CP^2$ if and only if $p$ is the maximum of a Markov triple $(a,b,p)$, and $q=\pm3b/a\pmod{p}$.  
Then in fact $q=\pm3u$, where $u$ is the characteristic number of the Markov triple.

For $n>1$, the odd Fibonacci number $F(2n+1)$ is the maximum of the Markov triple $(1,F(2n-1),F(2n+1))$, from which it follows that $B_{F(2n+1),F(2n-3)}$ embeds symplectically.  Also note that the characteristic number of this Markov triple is $F(2n-1)$, and $F(2n-1)^2\equiv-1\pmod {F(2n+1)}$.

Then $B_{F(2n+1),F(2n-1)}$ embeds symplectically if and only if $F(2n+1)$ is the maximum of another Markov triple $(a,b,F(2n+1))$, and $F(2n-1)=\pm3u$, where $u$ is the characteristic number of the triple $(a,b,F(2n+1))$.  This would imply that
$$-1\equiv F(2n-1)^2\equiv9u^2\equiv -9\pmod {F(2n+1)}.$$
The only odd Fibonacci numbers which divide 8 are $F(1)=1$ and $F(3)=2$, so we conclude that $n=1$.

Finally, $F(3)=2$ is the maximum of the Markov triple $(1,1,2)$ and $B_{F(3),F(1)}=B_{2,1}$ does embed sympectically.
\end{proof}

\begin{proof}[Proof of \Cref{thm:CP2}.]
As noted in the proof of \Cref{lem:nonsymp}, the Markov triple $(1,1,2)$ gives rise to an embedding of $B_{F(3),F(1)}=B_{2,1}$ in $\CP^2$.  Suppose now that $n>1$.
Induction using \eqref{eq:Fib} yields the Hirzebruch-Jung continued fraction expansion
$$\frac{F(2n+1)}{F(2n-1)}=[3^{n-1},2].$$
Now using \cite[Lemma 3.1]{balls} we have
$$\frac{F(2n+1)^2}{F(2n+1)F(2n-1)-1}=[3^{n-1},5,3^{n-2},2].$$

These continued fractions may be used to describe the surface $\Delta_{F(2n+1),F(2n-1)}$, as described in \cite{balls}.

The proof that $\Delta_{F(2n+1),F(2n-1)}$ is a sublevel surface of $P_{+}$ is a minor modification of the proof of \cite[Theorem 5]{balls}.  We  refer the reader to that source for details.  

Consider the first diagram shown in Figure \ref{fig:CP2}.  This represents a surface $\Sigma$ bounded by the unknot, which we claim is $P_+$.  Note first that the band move corresponding to the blue band labelled $0$ converts the diagram to one of $\Delta_{F(2n+1),F(2n-1)}$, which is the slice disk described by Casson and Harer \cite{ch} for the two-bridge knot $S(F(2n+1)^2,F(2n+1)F(2n-1)-1)$.
This shows that $\Delta_{F(2n+1),F(2n-1)}$ is a sublevel surface of the surface $\Sigma$.  It remains to see that $\Sigma$ is the unknotted M\"{o}bius band $P_{+}$ whose double branched cover is $\CP^2$ minus a 4-ball.

\begin{figure}[htbp]
\centering
\includegraphics[width=\textwidth]{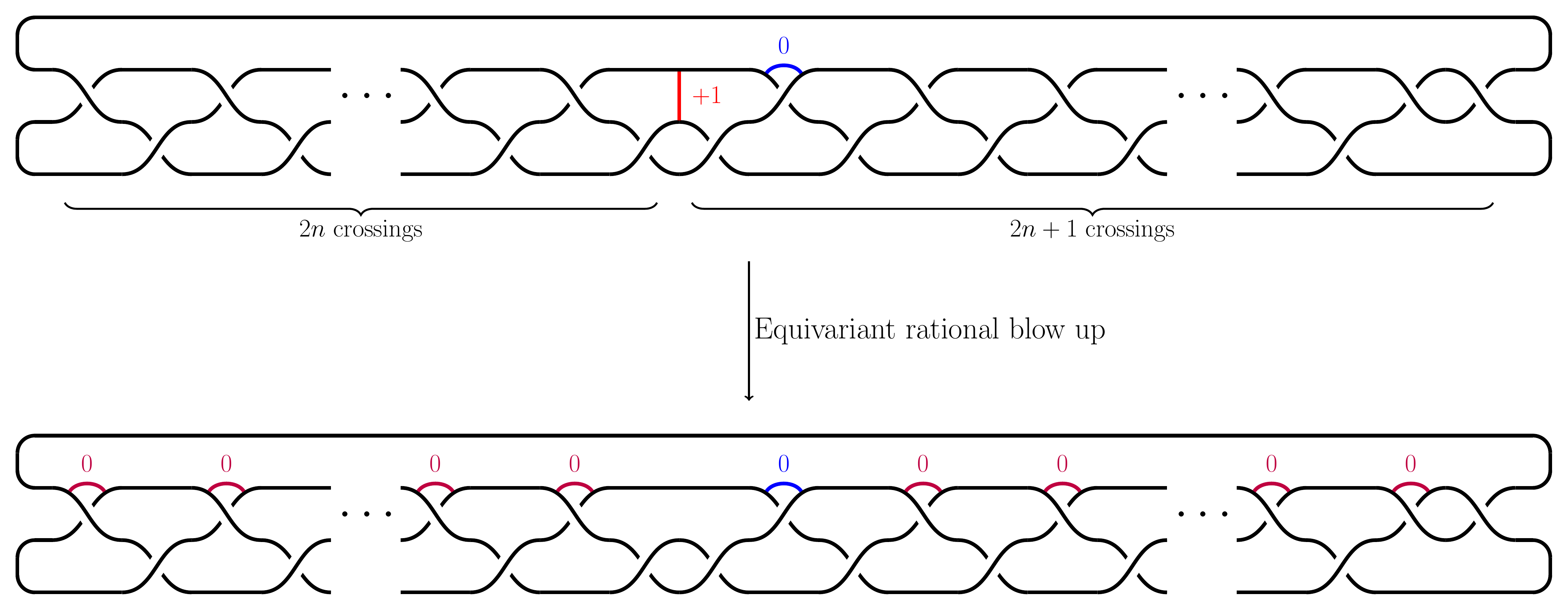}
\begin{narrow}{0.3in}{0.3in}
\caption
{{\bf The slice disk $\Delta_{F(2n+1),F(2n-1)}$ as a sublevel surface of $P_+$, and the resulting equivariant rational blow up.}  Numbers beside bands give the signed count of half-twists or crossings.}
\label{fig:CP2}
\end{narrow}
\end{figure}

Figure \ref{fig:CP2slides} shows a sequence of isotopies and band slides converting $\Sigma$ to $P_+$ in the first case of interest which is $n=2$.  
Taking double branched covers we see that $B_{5,2}$ admits a smooth embedding in $\CP^2$.  The proof for $n>2$ follows by an induction argument involving band slides similar to those in Figure \ref{fig:CP2slides}.  The inductive step is shown in Figure \ref{fig:CP2induction}.

\begin{figure}[htbp]
\centering
\includegraphics[width=\textwidth]{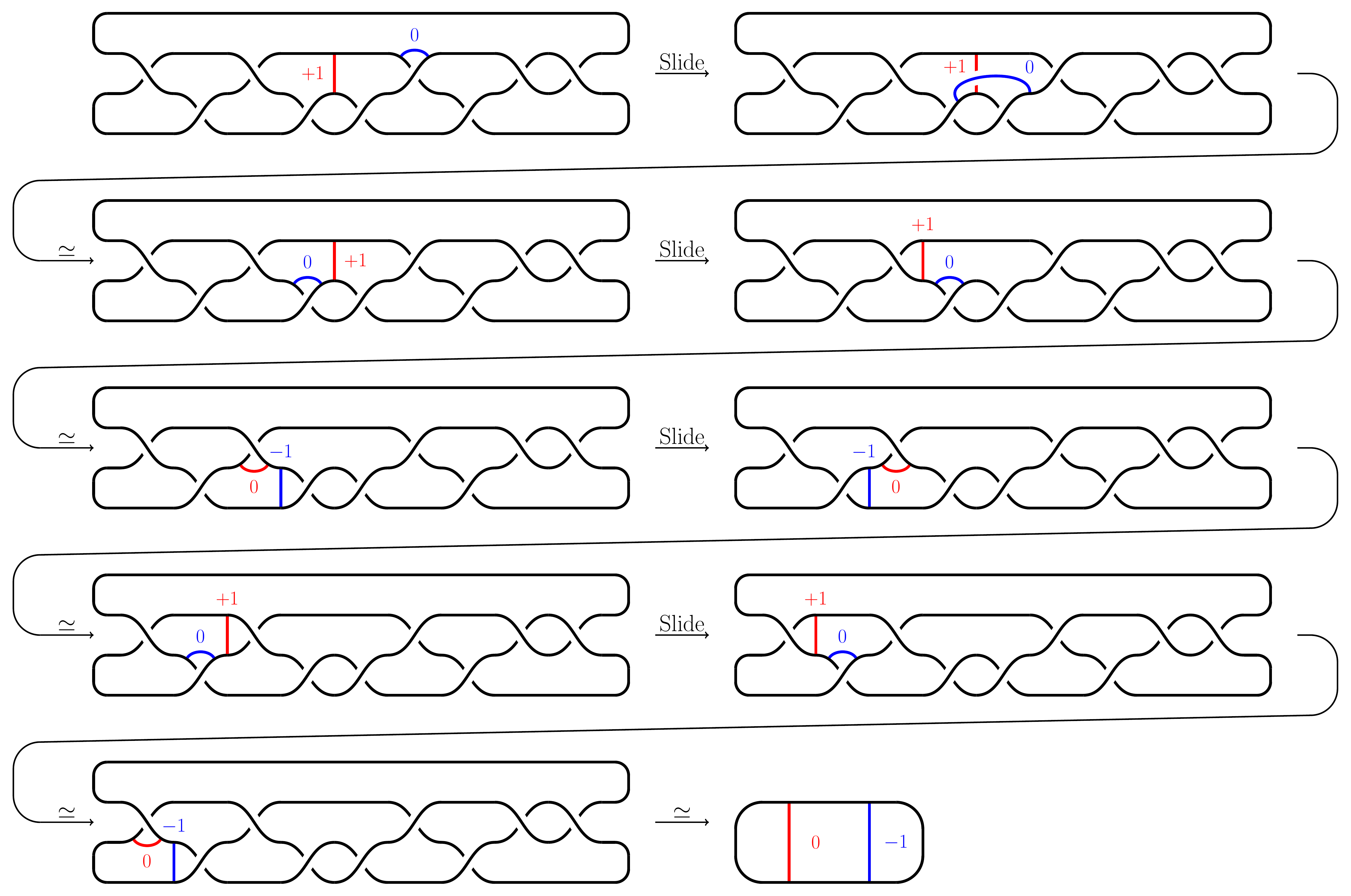}
\begin{narrow}{0.3in}{0.3in}
\caption
{{\bf The slice disk $\Delta_{5,2}$ is a sublevel surface of $P_+$.}}
\label{fig:CP2slides}
\end{narrow}
\end{figure}

Recall that an embedding of $B_{p,q}$ in a 4-manifold $Z$ is called simple if the resulting rational blow up of $Z$ is obtainable by a sequence of ordinary blow ups.  The proof that the embeddings described above are simple follows as in \cite[Proposition 5.1]{balls}; we again refer the reader to \cite{balls} for more details on equivariant rational blow up, and to \Cref{sec:obst} for a description of rational blow up.   We describe here a slightly shorter version of the proof at the level of double branched covers.  The second diagram in Figure \ref{fig:CP2} represents the surface in the 4-ball pushed in from the black surface of the two-bridge diagram shown, using a chessboard colouring in which the unbounded region is white.  The rational blow up of $\CP^2$, minus a 4-ball, is the double cover $X$ of the 4-ball branched along this black surface, which in turn is the plumbing of disk bundles over $S^2$ corresponding to the linear graph with weights
$$(-3)^{n-1},-2,-1,(-3)^{n-2},-2,$$
where $(-3)^m$ denotes $-3$ repeated $m$ times.
A sequence of $-1$ blow downs reduces this to the linear plumbing with weights $-3$ and $0$, which is diffeomorphic to $ \CP^2 \#\overline{\CP^2}$, again minus a ball.  It follows that
$$X\cong  \CP^2 \#(2n-1)\overline{\CP^2}.$$
Together with \Cref{lem:nonsymp}, this completes the proof of \Cref{thm:CP2}.\end{proof}

\begin{figure}[htbp]
\centering
\includegraphics[width=\textwidth]{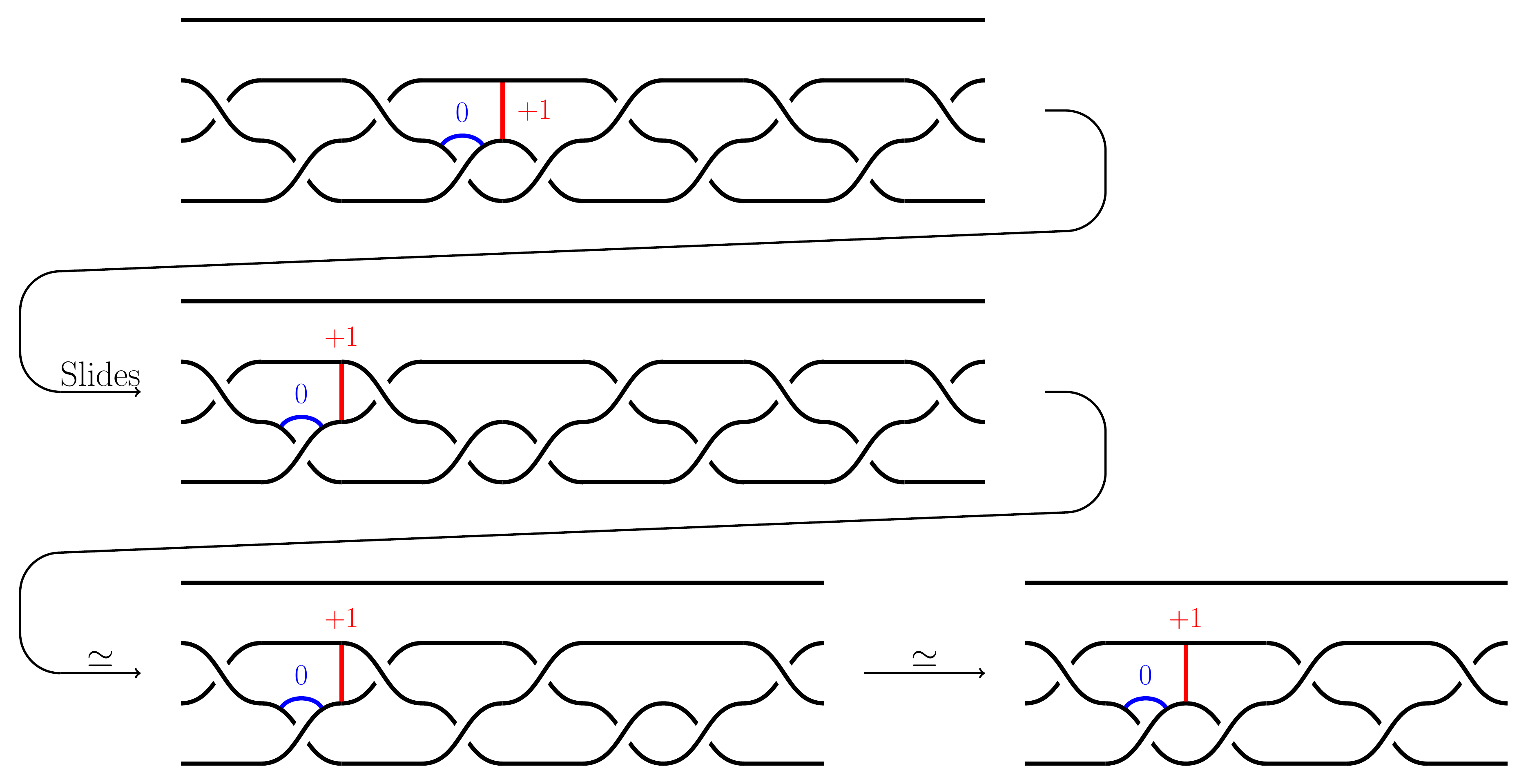}
\begin{narrow}{0.3in}{0.3in}
\caption
{{\bf The inductive step.} The band slides are similar to those shown in Figure \ref{fig:CP2slides}.  This shows how to transform the first diagram in Figure \ref{fig:CP2} with $n=k$, to the same diagram with $n=k-1$.}
\label{fig:CP2induction}
\end{narrow}
\end{figure}


\section{An obstruction from Donaldson's diagonalisation theorem}
\label{sec:obst}

In this section, we derive a lattice embedding obstruction to smoothly embedding a rational homology ball bounded by a lens space, or a disjoint union of such, in $\CP^2$.  We begin by setting some conventions and terminology.

All homology and cohomology groups in this section have integer coefficients.  Recall that if $X$ is a smooth 4-manifold, possibly with boundary, then its intersection lattice $\Lambda_X$ consists of the free abelian group $H_2(X)/\tors$ together with the symmetric bilinear intersection pairing.  The term lens space will be used here to refer to $L(p,q)$ with $p>q\ge1$; in particular not $S^3$ or $S^2\times S^1$.  Given integers $a_1,\dots,a_k$, the linear lattice $\Lambda(a_1,\dots,a_k)$ is defined to be the free abelian group with generators $v_1,\dots,v_n$, and with symmetric bilinear pairing given by
\begin{equation}
v_i\cdot v_j=
\begin{cases} a_i & \mbox{ if $i=j$;} \\ -1 & \mbox{ if $|i-j|=1$;}\\0 & \mbox{ if $|i-j|>1$.}\end{cases}
\label{eq:linlat}
\end{equation}
As this is the lattice associated to a weighted linear graph, we often refer to the generators $v_1,\dots,v_k$ as vertices.
Recall that a lens space $L(p,q)$ is the boundary of a plumbing $C$ of disk bundles over spheres determined by the weighted linear graph with weights $a_1,\dots,a_k\ge2$ where 
$$\frac{p}{p-q}=[a_1,a_2,\ldots,a_k]
:=a_1-\frac{1}{a_2-\raisebox{-3mm}{$\ddots$
\raisebox{-2mm}{${-\frac{1}{\displaystyle{a_k}}}$}}}.$$
The intersection lattice of $C$ is then $\Lambda(a_1,\dots,a_k)$.

Let $B$ be a rational homology ball with lens space boundary.  Given  an embedding $B\hookrightarrow \CP^2$, we let $M$ be the complement $\CP^2\setminus B$ and ``rationally blow up'' to obtain the closed positive-definite manifold $M\cup C$, where $C$ is the positive-definite plumbed manifold bounded by $\partial B$.  Donaldson's diagonalisation theorem then implies the existence of a lattice embedding
\begin{equation}\Lambda_M\oplus\Lambda_C\hookrightarrow \zz^m,
\label{eq:Don1}
\end{equation}
where $\Lambda_M$ and $\Lambda_C$ are the intersection lattices of $M$ and $C$ respectively, and $m$ is the sum of their ranks.

The reader familiar with the use of such lattice obstructions will note that since $M$ is a submanifold of $\CP^2$, and since $Y=\partial B$ bounds a rational ball, each of $\Lambda_M$ and $\Lambda_C$ admit finite-index embeddings in diagonal unimodular lattices, so that an embedding as in \eqref{eq:Don1} must in fact exist, with the first factor embedding in $\zz$ and the second in the orthogonal $\zz^{m-1}$.  We will show that simple topological considerations place further restrictions on the lattice embedding in \eqref{eq:Don1}, giving rise to a useful obstruction, which also extends to the case of an embedding of a disjoint union of rational balls.

\begin{lemma}
\label{lem:Zhom}
Let $B_i$ be rational homology balls bounded by lens spaces for $i=1,\dots,n$, and suppose that the disjoint union $\bigsqcup_{i}B_i$ embeds smoothly in $\CP^2$.  Then the complement $M=\CP^2\setminus\bigsqcup_{i}B_i$ has $H_1(M;\zz)=0$ and $H_2(M;\zz)\cong\zz$.
\end{lemma}

\begin{proof} We use the Mayer-Vietoris sequence and induction.  The base case is $n=0$ and $M=\CP^2$.

Now suppose $M'=\CP^2\setminus\bigsqcup_{i=1}^{n-1}B_i$ has $H_1(M';\zz)=0$ and $H_2(M';\zz)\cong\zz$.  Then $$M'=M\cup_Y B_n,$$
where $Y=L(p_n^2,q_n)$ has $H_1(Y;\zz)\cong\zz/p_n^2\zz$.  We have $H_2(B_n;\zz)=0$, since it is a torsion subgroup of $H_2(M';\zz)\cong\zz$; then from the long exact sequence of the pair $(B_n,Y)$, we have $H_1(B_n;\zz)\cong\zz/p_n\zz$. 
The Mayer-Vietoris sequence, with integer coefficients, shows that $H_2(M)$ is a finite-index subgroup of $\zz$, hence $H_2(M)\cong\zz$.  The same sequence shows that there is a surjection from $\zz/p_n^2\zz$ to $H_1(M)\oplus \zz/p_n\zz$, from which it follows that the latter direct sum is finite cyclic and also that cyclic summands of $H_1(M)$ have orders dividing $p_n$.  We conclude that $H_1(M)$ must be trivial.
\end{proof}

We recall the notion of rational blow up, and modify and generalise it for our convenience.  If a disjoint union $\bigsqcup_{i}B_i$ embeds smoothly in some 4-manifold $Z$, where each $B_i$ is a rational ball bounded by a lens space $L(p_i,q_i)$, then we may excise each $B_i$ and replace it by the positive-definite plumbed manifold $C_i$ bounded by $L(p_i,p_i-q_i)$ to obtain a new manifold 
$$X=M\cup C,$$
called the positive rational blow up of $Z$.  Here $M$ is the complement of $\bigsqcup_{i=1}^n B_i$ in $Z$, and $C$ is the disjoint union $\bigsqcup_{i=1}^n C_i$ of plumbed manifolds.  We assume that all weights in each plumbing $C_i$ are at least 2.

\begin{proposition}
\label{prop:Don}
Let $B_i$ be rational homology balls bounded by lens spaces 
for $i=1,\dots,n$, and suppose that the disjoint union $\bigsqcup_{i}B_i$ embeds smoothly in $\CP^2$.  Let $X=M\cup C$ be the resulting positive rational blow up of $\CP^2$.
Then there exists a finite-index lattice embedding
\begin{equation}
\label{eq:Don2}
\Lambda_M\oplus\Lambda_C\hookrightarrow\zz^m,
\end{equation}
such that each unit vector $e\in\zz^m$ has nonzero pairing with each of $\Lambda_M$ and $\Lambda_C$.  Moreover the image of the generator of $\Lambda_M$ is a primitive vector in $\zz^m$.
\end{proposition}

\begin{remark}
Let $A$ be the matrix of the embedding in \eqref{eq:Don2} in terms of a basis $v_1,\dots,v_m$ for $\Lambda_X$, where $v_1\in\Lambda_M$ and $v_2,\dots,v_m\in\Lambda_C$, and an orthonormal basis for $\zz^m$.  Then the proposition states that each row of $A$ has at least two nonzero entries including one in the first column, and also that the  entries of the first column of $A$, which are all nonzero, have no common divisor.

The known embeddings mentioned earlier in this section each give rise to a block diagonal matrix $A$ which does not satisfy the condition in the proposition.
\end{remark}
\begin{proof}[Proof of \Cref{prop:Don}.]
Let $Y$ denote the union of lens spaces which is the common boundary of $M$ and $C$.  Let $e$ be a unit vector in $\Lambda_X$.  We may write 
$$e=e_M+e_C,$$ where
$e_M\in H_2(M,Y)$ and $e_C\in H_2(C,Y)$.  There are no unit vectors in $\Lambda_M$, which is a rank one lattice whose generator squared is the order of the first homology of $Y$.  There are also no unit vectors in $\Lambda_C$ since we assumed all weights in each plumbing are at least 2.  It follows that  $e_M$ and $e_C$ are both nonzero.

Since $H_1(M)=0$ by \Cref{lem:Zhom}, all homology groups of $M$ are in fact torsion-free by standard arguments using universal coefficients, Poincar\'{e}-Lefschetz duality, and the long exact sequence of the pair.  It follows that the second homology group $H_2(M)$ is the underlying group of the lattice $\Lambda_M$, and the relative homology group $H_2(M,Y)$ is the underlying group of the dual lattice $\Lambda_M\!^*$ via the universal coefficient theorem.  Then since $\Lambda_M$ is positive definite, we see that an element of $H_2(M,Y)$ is nonzero if and only if it has nonzero intersection with some element in $H_2(M)$.  Thus in particular $e_M$ and also $e$ has nonzero intersection with some element of $H_2(M)$.
The same argument applies to $\Lambda_C$, so that $e_C$ and also $e$ has nonzero pairing with some element of $H_2(C)$ which is the underlying group of $\Lambda_C$.

Finally let $v$ denote the image in $H_2(X)$ of the generator of $\Lambda_M$, and suppose that $v=kw$ for some $k\in\nn$ and $w\in H_2(X)$.  As above we write $w=w_M+w_C$ and we conclude that $w_C=0$ since it has zero pairing with all of $\Lambda_C$.  This implies $w=w_M\in\Lambda_M$, but then $k=1$ since $v$ is the generator.
\end{proof}

In what follows we study lattice embeddings $\Lambda\hookrightarrow\zz^m$ up to lattice automorphisms of $\zz^m$, or in other words, up to reordering of the orthonormal basis $e_1,\dots,e_m$, and/or changing signs of some orthonormal basis elements.  Embeddings of linear lattices all of whose weights are 2 or 3 are very restricted, since up to $\Aut(\zz^m)$, vectors $v\in\zz^m$ with $v\cdot v=2$ or $v\cdot v=3$ take the form $v=e_1+e_2$ or $v=e_1+e_2+e_3$.

\begin{example}
\label{ex:B31}
The rational ball $B_{3,1}$ does not embed smoothly in $\CP^2$.
\end{example}
\begin{proof}
The boundary of $B_{3,1}$ is the lens space $L(9,2)$, which also bounds the positive-definite plumbing $C$ with weights $[2,2,2,3]$.  Let $v_2,\dots,v_5$ be the generators of the linear lattice $\Lambda_C=\Lambda(2,2,2,3)$ as in \eqref{eq:linlat},
and let $v_1$ be the generator of the rank one lattice $\Lambda_M=\Lambda(9)$.  Let $e_1,\dots, e_5$ be an orthonormal basis for $\zz^5$.  There is, up to lattice automorphisms of $\zz^5$, a unique embedding
$$\Lambda_M\oplus\Lambda_C\hookrightarrow\zz^5;$$
this takes $v_1$ to $3e_1$, $v_i$ to $-e_{i}+e_{i+1}$ for $2\le i\le4$, and  $v_5$ to $e_2+e_3+e_4$.  This does not satisfy the conditions of \Cref{prop:Don}, since $e_i$ has zero pairing with $\Lambda_M$ for $i>1$ and $e_1$ has zero pairing with $\Lambda_C$.
\end{proof}

\begin{lemma}
\label{lem:Cemb}
Up to $\Aut(\zz^m)$, there are precisely two ways to embed the linear lattice $\Lambda(2,2,2)$ in $\zz^m$, where $m\ge4$.  The first has image in a $\zz^3$ sublattice of $\zz^m$, and its orthogonal complement in this sublattice is the zero sublattice.  The second has image in a $\zz^4$ sublattice, and its orthogonal complement in $\zz^4$ is spanned by a vector $w$ with $w\cdot w=4$.

Let $n>1$ and let $\Lambda$ denote the linear lattice $\Lambda(3^{n-1},2,2,3^{n-1},2)$, with rank $r=2n+1$.  Up to $\Aut(\zz^m)$, there are precisely three ways to embed $\Lambda$ in $\zz^m$, where $m\in\nn$ is sufficiently large.   The first has image in a $\zz^r$ sublattice, and its orthogonal complement in this sublattice is the zero sublattice.  The second has image in a $\zz^{r+1}$ sublattice, and its orthogonal complement in $\zz^{r+1}$ is spanned by a vector $w$ with $w\cdot w=F(2n+1)^2$.  The third has image in a $\zz^{4n}$ sublattice, and its orthogonal complement in $\zz^{4n}$ contains no unit vectors.
\end{lemma}
\begin{proof} For the first case, we can either map the vertices of $\Lambda(2,2,2)$ to $-e_1+e_2,-e_2+e_3,e_1+e_2$ or to $-e_1+e_2,-e_2+e_3,-e_3+e_4$.  It is straightforward to see there are no other possibilities.

In the second case we begin by embedding the two adjacent vertices of weight two.  Up to automorphism of $\zz^m$, these are mapped to $-e_1+e_2$ and $-e_2+e_3$.  By inspection, the linear lattice $\Lambda(3,2,2,3)$, which is a sublattice of $\Lambda$, admits three possible embeddings up to symmetry as follows:
\begin{align}
\notag-&e_2-e_3-e_4,-e_1+e_2,-e_2+e_3,e_1+e_2-e_4;\\
\label{eq:Cemb}-&e_2-e_3-e_4,-e_1+e_2,-e_2+e_3,-e_3+e_4+e_5;\\
\notag\mbox{ or }\quad&\,e_1+e_4+e_5,-e_1+e_2,-e_2+e_3,-e_3+e_6+e_7.\\
\notag\end{align}

The first of these does not extend to an embedding of $\Lambda(3,2,2,3,2)$ or $\Lambda(3,2,2,3,3)$ so we discard it.  By a simple induction argument, the second of these extends uniquely to an embedding of $\Lambda(3^{n-1},2,2,3^{n-1})$ as follows:
$$-e_{2n-2}-e_{2n-1}-e_{2n},\dots,-e_4-e_5-e_6,-e_2-e_3-e_4,-e_1+e_2,$$
$$-e_2+e_3,-e_3+e_4+e_5,-e_5+e_6+e_7,\dots,-e_{2n-1}+e_{2n}+e_{2n+1}.$$
This can be extended to an embedding of $\Lambda$ in precisely two ways: we may map the additional weight two vertex to $e_{2n}-e_{2n+1}$ or to $-e_{2n+1}+e_{2n+2}$.
The first choice results in an embedding in $\zz^r$.  The second choice results in an embedding in $\zz^{r+1}$.  The orthogonal complement in $\zz^{r+1}$ has rank one and so is generated by a vertex $w$.  We may compute $w$ and hence its square directly or use the fact that $\Lambda$ is a primitive sublattice of $\zz^{r+1}$ with determinant $F(2n+1)^2$, which is therefore also the determinant of its rank one orthogonal complement.

Finally another simple induction argument shows that the third embedding  in \eqref{eq:Cemb} extends uniquely to $\Lambda(3^{n-1},2,2,3^{n-1})$, and also extends uniquely up to symmetry to give the following embedding of $\Lambda$:
\begin{equation}
\label{eq:Cemb3}
\begin{aligned}
e_{4n-7}&+e_{4n-4}-e_{4n-3},\dots,e_5+e_8-e_9,e_1+e_4-e_5,-e_1+e_2,-e_2+e_3,\\
-e_3&+e_6+e_7,-e_7+e_{10}+e_{11},\dots,-e_{4n-5}+e_{4n-2}+e_{4n-1},-e_{4n-1}+e_{4n}.
\end{aligned}
\end{equation}
We see that each of $e_1,\dots,e_{4n}$ appears in \eqref{eq:Cemb3}, and therefore has nonzero pairing with the image of this embedding.
\end{proof}

\begin{proof}[Proof of \Cref{thm:CP2disjoint}.]
For the duration of this proof, we denote by $B_n$ the rational ball $B_{F(2n+1),F(2n-1)}$, and by $C_n$ the positive-definite plumbed manifold with the same boundary as $B_n$, for each $n\in\nn$.  For $n=1$, the boundary of the rational ball $B_1=B_{2,1}$ is $L(4,1)$, and the plumbing $C_1$ has weights $[2,2,2]$.  For $n>1$, $C_n$ is the plumbing with weights $[3^{n-1},2,2,3^{n-1},2]$, as may be seen using \cite[Lemma 3.1]{balls}.

Suppose first that $B_1\sqcup B_n$ embeds smoothly in $\CP^2$.   Let $r_1=3$ and $r_2$ denote the ranks of $\Lambda_{C_1}$ and $\Lambda_{C_n}$ respectively.
By \Cref{prop:Don}, there is a finite-index lattice embedding
$$\Lambda_M\oplus\Lambda_{C_1}\oplus\Lambda_{C_n}\hookrightarrow\zz^m,$$
where $m=r_1+r_2+1=r_2+4$.
By \Cref{lem:Cemb}, the restriction of this to $\Lambda_{C_1}$ is either contained in a $\zz^3$ or is contained in a $\zz^4$, spanned by $e_1,\dots,e_4$ say, with orthogonal complement spanned by a vector $w$ of self-pairing 4.  Since the image of the generator of $\Lambda_M$ is orthogonal to the image of $\Lambda_{C_1}$ and has nonzero pairing with every unit vector in $\zz^m$ by \Cref{prop:Don}, it must be the second possibility.  The image of $\Lambda_{C_2}$ lies in the orthogonal complement to that of $\Lambda_{C_1}$.  If it is contained in the span of $e_5,\dots,e_m$ then this is a finite-index embedding in $\zz^{r_2}$ which again contradicts the fact that the image of the generator of $\Lambda_M$  has nonzero pairing with every unit vector.  Thus at least one vertex of $\Lambda_{C_2}$ contains a nonzero multiple of $w$.  This vertex then has self-pairing greater than that of $w$, contradicting the fact that the vertices of $\Lambda_{C_2}$ all have self-pairing 2 or 3.

We next suppose that $B_k\sqcup B_n$ embeds smoothly in $\CP^2$ with $n\ge k>1$.   Let $r_1=2k+1$ and $r_2=2n+1$ denote the ranks of $\Lambda_{C_k}$ and $\Lambda_{C_n}$ respectively.
By \Cref{prop:Don}, there is a finite-index lattice embedding
$$\Lambda_M\oplus\Lambda_{C_k}\oplus\Lambda_{C_n}\hookrightarrow\zz^m,$$
where $m=r_1+r_2+1=2k+2n+3$.

Arguing as in the previous case, we see that the restriction of this embedding to $\Lambda_{C_k}$ (respectively $\Lambda_{C_n}$) cannot have image in either $\zz^{r_1}$ or $\zz^{r_1+1}$ (respectively $\zz^{r_2}$ or $\zz^{r_2+1}$).  By \Cref{lem:Cemb}, this leaves the possibility that the restriction to $\Lambda_{C_k}$ lies in a $\zz^{4k}$ sublattice, and similarly the restriction to $\Lambda_{C_n}$ lies in a $\zz^{4n}$ sublattice, in both cases with the orthogonal complement in said sublattice containing no unit vectors.
In particular we have
$$4n\le 2k+2n+3,$$
and hence $n$ is either $k$ or $k+1$.

If $n=k+1$, we have $m=4n+1$.  
Up to $\Aut(\zz^m)$, we may suppose that the $\zz^{4n-4}$ sublattice containing the image of $\Lambda_{C_k}$ includes the vectors $-e_1+e_2,-e_2+e_3$ as the image of the two adjacent weight two vertices.  The $\zz^{4n}$ sublattice of $\zz^{4n+1}$ containing the image of $\Lambda_{C_n}$ has to intersect the $\zz^3$ sublattice spanned by $e_1,e_2,e_3$ nontrivially.  This means that some vertex of $\Lambda_{C_n}$ maps to a vector of the form $v+a(e_1+e_2+e_3)$, where $v$ is a nonzero vector in the span of $e_4,\dots,e_m$ and $a\ne0$, noting that the image of this vertex is orthogonal to $-e_1+e_2,-e_2+e_3$ and has pairing $-1$ with a neighbouring vertex.  This contradicts the fact that all vertices in $\Lambda_{C_n}$ have weight 2 or 3.

Finally if $n=k$ then $m=4n+3$.  We keep the notation $\Lambda_{C_n}$ and $\Lambda_{C_k}$ to distinguish the two copies of $\Lambda_{C_n}$.  We may suppose that the $\zz^{4n}$ sublattice containing the image of $\Lambda_{C_k}$ is the span of $e_1,\dots,e_{4n}$, and that it includes the vectors $-e_1+e_2,-e_2+e_3$ as the image of the two adjacent weight two vertices.   Arguing as in the case $n=k+1$, the image of $\Lambda_{C_n}$ has to be orthogonal to the span of $e_1, e_2, e_3$, and so is contained in the span of $e_4,\dots,e_m$.  We may also suppose that the two adjacent weight two vertices in $\Lambda_{C_n}$ map to $-e_{4n+1}+e_{4n+2},-e_{4n+2}+e_{4n+3}$.
We consider the image of $\Lambda_{C_n}$ and $\Lambda_{C_k}$ under the projection to $\zz^{4n-3}$ spanned by $e_4,e_5,\dots,e_{4n}$.  From \eqref{eq:Cemb3} we see that each of these is isomorphic to the orthogonal direct sum $\Lambda(3^{n-2},2)\oplus\Lambda(2,3^{n-2},2)$, and has rank $2n-1$.   This leads to a contradiction, since it is not possible to orthogonally embed two lattices of rank $2n-1$ in $\zz^{4n-3}$.
\end{proof}


\clearpage

\bibliographystyle{amsplain}
\bibliography{balls}

\providecommand{\bysame}{\leavevmode\hbox to3em{\hrulefill}\thinspace}
\providecommand{\MR}{\relax\ifhmode\unskip\space\fi MR }
\providecommand{\MRhref}[2]{%
  \href{http://www.ams.org/mathscinet-getitem?mr=#1}{#2}
}
\providecommand{\href}[2]{#2}
\begin{thebibliography}{1}

\bibitem{aigner}
Martin Aigner, \emph{Markov's theorem and 100 years of the uniqueness
  conjecture}, Springer, Cham, 2013.

\bibitem{ch}
Andrew~J. Casson and John~L. Harer, \emph{Some homology lens spaces which bound
  rational homology balls}, Pacific J. Math. \textbf{96} (1981), no.~1, 23--36.

\bibitem{don}
S.~K. Donaldson, \emph{The orientation of {Y}ang-{M}ills moduli spaces and
  {$4$}-manifold topology}, J. Differential Geom. \textbf{26} (1987), no.~3,
  397--428.

\bibitem{es}
Jonathan~David Evans and Ivan Smith, \emph{Markov numbers and {L}agrangian cell
  complexes in the complex projective plane}, Geom. Topol. \textbf{22} (2018),
  no.~2, 1143--1180.

\bibitem{hp}
Paul Hacking and Yuri Prokhorov, \emph{Smoothable del {P}ezzo surfaces with
  quotient singularities}, Compos. Math. \textbf{146} (2010), no.~1, 169--192.

\bibitem{kollar}
J\'{a}nos Koll\'{a}r, \emph{Is there a topological {B}ogomolov-{M}iyaoka-{Y}au
  inequality?}, Pure Appl. Math. Q. \textbf{4} (2008), no.~2, Special Issue: In
  honor of Fedor Bogomolov. Part 1, 203--236.

\bibitem{NS}
Stefan Nemirovski and Kyler Siegel, \emph{Rationally convex domains and
  singular {L}agrangian surfaces in {$\Bbb{C}^2$}}, Invent. Math. \textbf{203}
  (2016), no.~1, 333--358.

\bibitem{balls}
Brendan Owens, \emph{Equivariant embeddings of rational homology balls}, Q. J.
  Math. \textbf{69} (2018), no.~3, 1101--1121.

\bibitem{psu}
H.~Park, D.~Shin, and G.~Urz{{\'u}}a, \emph{Simple embeddings of rational
  homology balls and antiflips}, arXiv:1904.04927, 2019.

\end{thebibliography}

\end{document}